\documentclass[smallextended,numbook,runningheads]{svjour3}
\usepackage{amsmath}
\smartqed  

\usepackage{amsfonts,amssymb}
\usepackage{epsfig}
\usepackage{booktabs}

\journalname{}
\date{ \phantom{b} \vspace{45mm}\phantom{e}}

\usepackage{color}

\def\d{{\rm d}}

\def\R{{\mathbb R}}

\def\bigo{{\mathcal O}}
\def\AA{\bigo\hskip-2.4pt\iota}
\def\ones{1\hskip-2.2pt{\mathrm l}}
\newcommand\calA{{\cal A}}
\newcommand\calB{{\cal B}}

\newcommand\calD{{\cal D}}

\newcommand\calS{{\cal S}}

\begin{document}

\title{Stability and convergence of time discretizations of quasi-linear evolution equations of Kato type}
\thanks{This work has been supported  by Deutsche Forschungsgemeinschaft, SFB 1173.}
\titlerunning{Time discretization of quasi-linear evolution equations of Kato type}

\author{Bal\'{a}zs Kov\'{a}cs \and Christian Lubich}
\institute{
Mathematisches Institut, Universit\"at T\"ubingen, Auf der Morgenstelle,
  D-72076 T\"ubingen, Germany.\\
  \email{\{Kovacs,Lubich\}@na.uni-tuebingen.de}}

\authorrunning{B. Kov\'{a}cs and Ch. Lubich}
\date{}

\maketitle
\begin{abstract} Semidiscretization in time is studied for a class of quasi-linear evolution equations in a framework due to Kato, which applies to symmetric first-order hyperbolic systems and to a variety of fluid and wave equations. In the regime where the  solution is sufficiently regular, we show stability and optimal-order convergence of the linearly implicit and fully implicit midpoint rules and of higher-order implicit Runge--Kutta methods that are algebraically stable and coercive, such as the collocation methods at Gauss nodes.

\end{abstract}

\section{Introduction}
\label{sect:intro}
In a very insightful paper published in 1975, Kato~\cite{Kato75} presents a concise framework for quasi-linear evolution equations in a Banach space, proves local well-posedness of the initial value problem within this framework and shows that the framework and results apply to a variety of quasi-linear partial differential equations. He lists symmetric hyperbolic systems of the first order, wave equations,
Korteweg--de Vries equation,
Navier--Stokes and Euler equations,
equations for compressible fluids, magnetohydrodynamic equations, coupled Maxwell and Dirac equations --- and adds ``etc.". Particularly noteworthy appears the application to symmetric hyperbolic systems in the sense of Friedrichs (in arbitrary space dimension), which is a large and fundamental class of problems.

While Kato's paper has been influential and highly cited in the analysis of nonlinear hyperbolic and dispersive partial differential equations, it has apparently gone unnoticed in the numerical literature for such equations. Kato's framework has been modified and generalized to further classes of partial differential equations, by himself and coauthors in \cite{HugKM,Kato1975cauchy} shortly after \cite{Kato75}, and by other researchers until recently, e.g., in \cite{dorfler2015local,diss_Muller}. To our knowledge, the only numerical paper related to Kato's framework is the recent work of Hochbruck \& Pa{\v{z}}ur \cite{HochbruckPazur} who study  the implicit Euler time discretization in a modified Kato framework that was developed by M\"uller \cite{diss_Muller} for dealing with a class of quasi-linear Maxwell equations. We acknowledge that it was \cite{HochbruckPazur} and \cite{diss_Muller} that led us to the present work.

Here we show that Kato's original framework from \cite{Kato75}, when restricted to Hilbert spaces (which are mostly used in the applications), combines remarkably well with the technique of ``energy estimates" for time discretizations, that is, with the use of positive definite and semi-definite bilinear forms for proving stability and error bounds. We show this first for the implicit midpoint rule in Section~3, and then (in Section~4) for implicit Runge--Kutta methods such as the Gauss and Radau IIA methods of arbitrary orders, which have the properties of  algebraic stability and coercivity, notions that are due to Burrage \& Butcher~\cite{BurB} and Crouzeix~\cite{Cro} (for algebraic stability) and to Crouzeix and Raviart \cite{CroR} (for coercivity); see also \cite{DekV,HairerWannerII}. Although these notions were developed and recognized as important properties in the context of stiff ordinary differential equations in the same decade in which Kato's paper appeared, it seems that no link between these analytical and numerical theories was made. With a delay of some decades, this is now done in the present paper --- in view of both, the perfectly fitting connection of the analytical framework and the numerical methods, and the undiminished significance of the considered evolution equations in applications.

We study only time discretization in this paper. Effects of truncation of an unbounded spatial domain and of space discretization are not considered here. Moreover, we work in a regime where a sufficiently regular solution exists. Of course, we are aware that shocks may develop in finite time in quasi-linear hyperbolic equations. Nevertheless, for many cases within the class of evolution equations considered (in particular, in problems of wave propagation and dispersive equations), regular solutions exist for sufficiently long times of interest (or even for all times), and it is then important to understand the mechanisms that yield stability and convergence of numerical discretizations.

%
%
%
%
%

\section{Kato's framework in a Hilbert space setting}
We consider a quasi-linear evolution equation (with $\dot{} = \d/\d t$)
\begin{equation} \label{Q}
\dot u + A(u)u = f(u)
\end{equation}
in the following setting, which is a Hilbert space version of the framework in Kato's paper \cite{Kato75}.
 Let $X$ and $Y$ be two real Hilbert spaces such that $Y$ is densely and continuously embedded in $X$. We denote the inner product on $X$ by $(\cdot,\cdot)$ and the norms on $X$ and $Y$ by
$$
|\cdot| = \|\cdot \|_X, \quad \|\cdot\| = \|\cdot \|_Y.
$$
For convenience we choose the norms such that $|y|\le \|y\|$ for all $y\in Y$.
We assume throughout this paper that for every $R>0$ the following assumptions are satisfied, with real numbers $M_R^A$, $M_R^B$, $L_R^A$, $L_R^B$, $\ell_R^X$, $\ell_R^Y$ depending only
on~$R$:

(K1) ({\it m-accretivity} \cite[Section V.10]{Kato-book}) For every $y\in Y$, the closed linear operator $A(y)$ on $X$  has the open left complex half-plane in the resolvent set and satisfies the  bound
\begin{equation}
\label{accretive}
(w,A(y)w) \ge 0 \qquad\text{for all }\ w\in D(A(y)).
\end{equation}
Moreover, the domain $D(A(y))$ contains the space $Y$, and there is the $Y$-locally uniform bound, for $y\in Y$ with $\|y\|\le R$,
\begin{equation}
\label{A-bound}
|A(y)w| \le M_R^A \, \|w\| \qquad\text{for all }\ w\in Y.
\end{equation}

(K2) ({\it Kato's commutator condition})
There exists an isometry $S:Y\to X$, self-adjoint as a linear operator on $X$, with the following property: For every $y\in Y$ with $\|y\|\le R$,
\begin{equation}
\label{S}
SA(y)S^{-1} = A(y) + B(y)
\end{equation}
(with equality of domains), where $B(y)$ is $Y$-locally uniformly bounded on $X$:
\begin{equation}
\label{B-bound}
|B(y)v| \le M_R^B \, |v| \qquad\text{for all }\ v\in X.
\end{equation}

(K3) ({\it $Y$-local Lipschitz conditions}) For all $y,\widetilde y \in Y$ with $\|y\|\le R$,
$\|\widetilde y\|\le R$,
\begin{align}
\label{Lip-A}
&\bigl|(A(y)-A(\widetilde y))w \bigr| \le L_R^A\, | y-\widetilde y |\, \|w\| \qquad\text{for all }\ w\in Y,
\\
\label{Lip-B}
&\bigl|(B(y)-B(\widetilde y))v \bigr| \le L_R^B\, \| y-\widetilde y \|\,\, |v| \qquad\, \text{for all }\ v\in X.
\end{align}

(K4) ({\it Semilinear term}) The function $f: X \to X$ is $Y$-locally Lipschitz-continuous in $X$ and $Y$:
For all $y,\widetilde y \in Y$ with $\|y\|\le R$,
$\|\widetilde y\|\le R$,
\begin{align}
\label{Lip-f-X}
|f(y)-f(\widetilde y)| &\le \ell_R^X\, | y-\widetilde y |,
\\
\label{Lip-f-Y}
\|f(y)-f(\widetilde y)\| &\le \ell_R^Y\, \| y-\widetilde y \|.
\end{align}

In \cite{Kato75}, Kato just assumes Banach spaces instead of Hilbert spaces, and he requires that $-A(y)$ is the generator of a  contraction semigroup on $X$. On a Hilbert space, this condition is equivalent to (K1) by the Lumer--Phillips theorem \cite[p.\,14]{Pazy}.

Under these conditions, Kato  \cite[Theorem 6]{Kato75} proves local existence and uniqueness of a solution  to \eqref{Q} in $C([0,\bar t];Y)\cap C^1([0,\bar t];X)$ (for some $\bar t>0$)  for initial data in $Y$.

He then proceeds to show that (K1)--(K4) are indeed satisfied for a wide variety of  quasi-linear partial differential equations, as listed in the introduction. In these applications, he has typically
\begin{equation}\label{sobolev}
X=L_2(\R^d),\ Y=H^s(\R^d), \ \text{and the isometry } S=(I-\Delta)^{s/2}
\end{equation}
for an exponent $s>0$ that is sufficiently large so that the $s$th-order Sobolev space $H^s$ is a Banach algebra ($s>d/2$).

Moreover, Kato \cite[Theorem 7]{Kato75} also gives a perturbation result. Here we give another perturbation result together with its simple proof, because its time-discrete versions will be important later in this paper.

Suppose $u(t)\in Y$ solves \eqref{Q} for $0\le t \le T$ and $u^\star(t)\in Y$ solves \eqref{Q} up to a defect $d(t)\in Y$ for $0\le t \le T$:
\begin{equation} \label{Q-d}
\dot u^\star + A(u^\star)u^\star = f(u^\star) + d.
\end{equation}

\begin{lemma}  \label{lem:pert-cont}
In the above situation, suppose that, for $0\le t \le T$,
$$
 \|u^\star(t)\|\le R \quad\text{ and }\quad Su^\star(t)\in Y \text{ with }\ \| Su^\star(t) \| \le K.
$$
 Then, there exists $\delta>0$, which depends only on $K$, $R$ and $T$ such that for perturbations satisfying
 $$
 \| u(0)- u^\star(0)\|^2 + \int_0^T \| d(s) \|^2 \, \d s \le \delta^2,
 $$
 the error $u-u^\star$ satisfies, for $0\le t \le T$,
\begin{align*}
\| u(t) - u^\star(t) \|^2 &\le C_Y  \Bigl(  \| u(0)- u^\star(0)\|^2 + \int_0^t \| d(s) \|^2 \, \d s \Bigr),
\\
| u(t) - u^\star(t) |^2 &\le C_X  \Bigl(  | u(0)- u^\star(0)|^2 + \int_0^t | d(s) |^2 \, \d s \Bigr) ,
\end{align*}
where $C_Y$  depends only on $K$, $R$ and $T$, and $C_X$ depends only on $R$
and~$T$.
\end{lemma}

\noindent
{\it Remark\/} In the situation~\eqref{sobolev} the condition $Su^\star(t)\in Y$ means $u^\star(t)\in H^{2s}$. This higher regularity is again ensured locally in time for initial values in $H^{2s}$, using Kato's theory with $2s$ in place of $s$.

\begin{proof} The error $e=u-u^\star$ satisfies the error equation
\begin{equation}\label{err-eq-cont}
\dot e + A(u)e = - \bigl( A(u)-A(u^\star) \bigr) u^\star + \bigl(f(u)-f(u^\star)\bigr) -d.
\end{equation}
Using \eqref{S} as $A(y)=S^{-1}A(y)S + S^{-1}B(y)S$, this becomes
\begin{align*}
&\dot e + S^{-1}A(u)Se + S^{-1}B(u)Se
\\
&= -S^{-1} \bigl( A(u)-A(u^\star) \bigr) Su^\star -
S^{-1} \bigl( B(u)-B(u^\star) \bigr) Su^\star +  \bigl(f(u)-f(u^\star)\bigr) -d.
\end{align*}
On applying the operator $S$ on both sides we thus have
\begin{align*}
&S\dot e + A(u)Se + B(u)Se
\\&
= - \bigl( A(u)-A(u^\star) \bigr) Su^\star -
\bigl( B(u)-B(u^\star) \bigr) Su^\star  + \bigl(Sf(u)-Sf(u^\star)\bigr) - Sd.
\end{align*}
Using the accretivity \eqref{accretive} and the bounds \eqref{B-bound}--\eqref{Lip-f-Y} and recalling that $S$ is an isometry between $Y$ and $X$, we therefore obtain, as long as $\| u(t) \|\le 2R$,
\begin{align*}
\tfrac12\tfrac{\d}{\d t} \| e \|^2 &= \tfrac12 \tfrac{\d}{\d t} | Se |^2 = (Se,S\dot e)
\\
&\le M_{2R}^B \|e\|^2 + L_{2R}^A \|e\|\,|e|\, \| Su^\star\| + L_{2R}^B \|e\|^2\, | Su^\star| +
\ell_{2R}^Y \|e\|^2 + \|e\|\, \|d\|
\\
&\le (M_{2R}^B + L_{2R}^A K + L_{2R}^B R + \ell_{2R}^Y +\tfrac1{2T}) \|e\|^2 + \tfrac T2 \|d\|^2,
\end{align*}
and the error bound in the $Y$-norm follows with Gronwall's inequality. Choosing $\delta$ so small that $C_Y\delta\le R$, the condition $\| u(t) \|\le 2R$ then remains satisfied for $0\le t \le T$. Taking in \eqref{err-eq-cont} the inner product with $e$ and using  \eqref{accretive}, \eqref{Lip-A}, \eqref{Lip-f-X} gives us
$$
\tfrac12 \tfrac{\d}{\d t} | e |^2 = (e,\dot e) \le (L_{2R}^AR+\ell_{2R}^X) |e|^2 + |e|\,|d|,
$$
and finally the Gronwall inequality yields the error bound in the $X$-norm.
\qed
\end{proof}

\section{Linearly implicit and fully implicit midpoint rules}
For the time discretization of \eqref{Q} we first consider variants of the implicit midpoint rule. For a positive stepsize $\tau$ and integers $n=0,1,2,\dots$, the solution $u(t)$ to \eqref{Q} with initial value $u_0$ is approximated at $t_n=n\tau$ by $u_n$, which is determined by
\begin{equation}
\label{mpr}
\frac{u_{n+1} - u_n}\tau + A(\widehat u_{n+1/2})\frac{u_{n+1}+u_n}2 = f(\widehat u_{n+1/2}).
\end{equation}
Here we set either
\begin{itemize}
\item[(FI)\!\!\!\!\!\!\!\!] $\qquad\widehat u_{n+1/2} = \frac{u_{n+1}+u_n}2\ \ $ for the fully implicit midpoint rule, or
\item[(LI)\!\!\!\!\!\!\!\!] $\qquad\widehat u_{n+1/2} = u_n + \tfrac12 (u_n-u_{n-1})\ \ $ for a linearly implicit midpoint rule.
\end{itemize}
In the latter case, we set $\widehat u_{1/2}=u_0$ in the first step.

For ease of presentation, we just consider constant stepsizes in this paper, but all our results generalize to variable stepsizes (with a bounded ratio of subsequent stepsizes) without any additional difficulty.

\subsection{Stability of the linearly implicit midpoint rule}

We begin with the stability analysis of the linearly implicit method, \eqref{mpr} with (LI) in the setting of Section~2. As is clear from the framework of Section~2, it is important to control the $Y$-norm of the numerical solution. This can be done {\it per se} (Lemma~\ref{lem:Y-bound-li}) or in comparison with the exact solution (Lemma~\ref{lem:pert-li}).

\begin{lemma} \label{lem:Y-bound-li}
Suppose that for all $k\le n$ we have $u_k\in Y$ with $\| u_k\| \le R$. Then there exist
 $\tau_R>0$ and $C_R,\gamma_R\ge 0$, which depend only on $R$ through the constants in conditions (K1)--(K4), such that for stepsizes $\tau\le\tau_R$ the linearly implicit midpoint rule \eqref{mpr} with (LI) has  a unique solution $u_{n+1}\in Y$. Moreover, this is bounded by
$$
\| u_{n+1} \| \le (1+C_R\tau) \| u_n \| + \tau \gamma_R.
$$
\end{lemma}

\begin{proof} Let us introduce the abbreviations
\begin{equation}\label{u-abbr}
u_{n+1/2} = \frac{u_{n+1} + u_n}2, \quad\  \dot u_{n+1/2} = \frac{u_{n+1} - u_n}\tau ,
\end{equation}
so that the numerical method \eqref{mpr} reads more concisely
$$
\dot u_{n+1/2} + A(\widehat u_{n+1/2}) u_{n+1/2} = f(\widehat u_{n+1/2}).
$$
With \eqref{S}, this is written equivalently as
$$
\dot u_{n+1/2} + S^{-1}A(\widehat u_{n+1/2}) Su_{n+1/2} + S^{-1}B(\widehat u_{n+1/2}) Su_{n+1/2}= f(\widehat u_{n+1/2}),
$$
where we note that $\| \widehat u_{n+1/2} \| = \| \frac32 u_n - \frac12 u_{n-1} \| \le 2R$. We apply $S$ to both sides of the equation and  obtain a linear equation in $X$ for $Su_{n+1}$ with the operator
$
\tau^{-1} I + A(\widehat u_{n+1/2}) + B(\widehat u_{n+1/2}),
$
which by (K1)  is invertible if $1/\tau > \|B(\widehat u_{n+1/2})\|$. In view of \eqref{B-bound} this is satisfied if $\tau M_{2R}^B <1$. Hence, under this stepsize restriction we have a unique solution $u_{n+1}\in Y$.

To derive the bound for $\| u_{n+1} \|$, we take the inner product with $Su_{n+1/2}$ in the equation. With the accretivity \eqref{accretive} and the bound \eqref{B-bound} we obtain
$$
(Su_{n+1/2},S\dot u_{n+1/2}) \le M_{2R}^B |Su_{n+1/2}|^2 + |Su_{n+1/2}| \, |Sf(\widehat u_{n+1/2})|.
$$
The term on the left-hand side is
$$
(Su_{n+1/2},S\dot u_{n+1/2}) = \frac1{2\tau}\, \bigl( |Su_{n+1}|^2-|Su_n|^2 \bigr) =
\frac1{2\tau}\, \bigl( \|u_{n+1}\|^2-\|u_n\|^2 \bigr).
$$
On the right-hand side we note
\begin{align*}
|Sf(\widehat u_{n+1/2})| = \|f(\widehat u_{n+1/2})\| &\le \| f(\widehat u_{n+1/2})- f(0)\| + \| f(0) \|
\\
&\le \ell_{2R}^Y\, 2R + \| f(0) \| =:c_{2R}.
\end{align*}
Hence we obtain
$$
\| u_{n+1} \|^2 \le \| u_n \|^2 + \tau M_{2R}^B \| u_{n+1/2} \|^2 + \tau c_{2R}  \| u_{n+1/2} \|.
$$
Using that $\| u_{n+1/2} \|^2 \le \frac12 (\| u_{n+1} \|^2 + \| u_{n} \|^2)$, the result follows.
\qed
\end{proof}

Suppose that $u_n\in Y$ solves \eqref{mpr} with (LI) for $0\le n\tau \le T$, and $u_n^\star\in Y$ (which will later be taken as the exact solution value $u(t_n)$) solves \eqref{mpr} with (LI) up to a defect $d_{n+1/2}\in Y$ for $0\le n\tau \le T$:
\begin{equation}
\label{mpr-star}
\frac{u_{n+1}^\star - u_n^\star}\tau + A(\widehat u_{n+1/2}^\star)\frac{u_{n+1}^\star+u_n^\star}2 = f(\widehat u_{n+1/2}^\star) +d_{n+1/2},
\end{equation}
with $\widehat u_{n+1/2}^\star = u_n^\star + \tfrac12 (u_n^\star-u_{n-1}^\star)$.

We then have the following time-discrete version of Lemma~\ref{lem:pert-cont}.

\begin{lemma}  \label{lem:pert-li}
In the above situation, suppose that, for $0\le n\tau \le T$,
$$
 \|u_n^\star\|\le R \quad\text{ and }\quad Su_n^\star\in Y \text{ with }\ \| Su_n^\star \| \le K.
$$
 Then, there exist  $\overline\tau>0$ and $\delta>0$, which depend only on $K$, $R$ and $T$, such that for stepsizes $\tau\le\overline\tau$ and perturbations satisfying
 $$
 \| u_0- u_0^\star\|^2 +\tau\sum_{0\le n\tau\le T} \| d_{n+1/2} \|^2  \le \delta^2,
 $$
 the error satisfies, for $0\le n\tau \le T$,
\begin{align*}
\| u_n- u_n^\star \|^2 &\le C_Y  \Bigl(  \| u_0- u_0^\star\|^2 +\tau\sum_{k=0}^{n-1} \| d_{k+1/2} \|^2 \Bigr),
\\
| u_n - u_n^\star |^2 &\le C_X  \Bigl(  | u_0- u_0^\star |^2 +\tau\sum_{k=0}^{n-1} | d_{k+1/2} |^2 \Bigr) ,
\end{align*}
where $C_Y$  depends only on $K$, $R$ and $T$, and $C_X$ depends only on $R$
and~$T$.
\end{lemma}

\begin{proof} The proof transfers the arguments of the proof of Lemma~\ref{lem:pert-cont} to the discrete case.
With the error $e_n=u_n-u_n^\star$ we associate the abbreviations (cf.\,\eqref{u-abbr})
$$
e_{n+1/2} = \frac{e_{n+1} + e_n}2, \quad\  \dot e_{n+1/2} = \frac{e_{n+1} - e_n}\tau, \quad\
\widehat e_{n+1/2} = e_n + \tfrac12(e_n-e_{n-1}).
$$
We have the error equation
\begin{align}\nonumber
\dot e_{n+1/2} + A(\widehat u_{n+1/2})e_{n+1/2} = &- \bigl( A(\widehat u_{n+1/2})-A(\widehat u_{n+1/2}^\star) \bigr) u_{n+1/2}^\star
\\& + \bigl(f(\widehat u_{n+1/2})-f(\widehat u_{n+1/2}^\star)\bigr) -d_{n+1/2}.
\label{err-eq-li}
\end{align}
Using \eqref{S} as $A(y)=S^{-1}A(y)S + S^{-1}B(y)S$ and  applying the operator $S$ on both sides we thus have
\begin{align*}
&S\dot e_{n+1/2} + A(\widehat u_{n+1/2})Se_{n+1/2} + B(\widehat u_{n+1/2})Se_{n+1/2}
\\&
=  \bigl( A(\widehat u_{n+1/2})-A(\widehat u_{n+1/2}^\star) \bigr) Su_{n+1/2}^\star +
\bigl( B(\widehat u_{n+1/2})-B(\widehat u_{n+1/2}^\star) \bigr) Su_{n+1/2}^\star
\\
&\quad+ \bigl(Sf(\widehat u_{n+1/2})-Sf(\widehat u_{n+1/2}^\star)\bigr) - Sd_{n+1/2}.
\end{align*}
Using the accretivity \eqref{accretive} and the bounds \eqref{B-bound}--\eqref{Lip-f-Y} and recalling that $S$ is an isometry between $Y$ and $X$, we therefore obtain, as long as $\| u_n \|\le 2R$,
\begin{align*}
&\tfrac1{2\tau}\bigl( \| e_{n+1} \|^2 - \| e_n \|^2 \bigr) = \tfrac1{2\tau}\bigl( | Se_{n+1} |^2 - | Se_n |^2 \bigr) = (Se_{n+1/2},S\dot e_{n-1/2})
\\
&\le M_{2R}^B \|e_{n+1/2}\|^2
 \\
&\quad+ L_{2R}^A \|e_{n+1/2}\|\,|\widehat e_{n+1/2}|\, \| Su_{n+1/2}^\star\|
 +
L_{2R}^B \|e_{n+1/2}\|\,\|\widehat e_{n+1/2}\|\, | Su_{n+1/2}^\star|
 \\
&\quad+
\ell_{2R}^Y \|e_{n+1/2}\|\, \|\widehat e_{n+1/2}\|
+ \|e_{n+1/2}\|\, \|d_{n+1/2}\|.
\end{align*}
The right-hand side is bounded by $C_{K,R} (\| e_{n+1} \|^2 + \| e_n \|^2 + \|e_{n-1}\|^2) + \|d_{n+1/2}\|^2$
and the stated error bound in the $Y$-norm then follows on summing up and using a discrete Gronwall inequality. Choosing $\delta$ such that $C_Y\delta\le R$, the condition $\| u_n \|\le 2R$ then remains satisfied for $0\le n\tau \le T$. Taking in \eqref{err-eq-li} the inner product with $e_{n+1/2}$ and using  \eqref{accretive}, \eqref{Lip-A}, \eqref{Lip-f-X} gives us
$$
\tfrac1{2\tau}\bigl( | e_{n+1} |^2 - | e_n |^2 \bigr) = (e_{n+1/2},\dot e_{n+1/2}) \le
C_R (| e_{n+1} |^2 + | e_n |^2 + |e_{n-1}|^2) + |d_{n+1/2}|^2,
$$
and finally a discrete Gronwall inequality yields the stated error bound in the $X$-norm.
\qed
\end{proof}

\subsection{Existence and stability for the fully implicit midpoint rule}

\begin{lemma}\label{lem:Y-bound-fi} The statement of Lemma~\ref{lem:Y-bound-li} is also valid for the fully implicit midpoint rule, \eqref{mpr} with (FI).
\end{lemma}

\begin{proof} The proof transfers the existence proof for \eqref{Q} in \cite{Kato75} to the time discretization. We consider the fixed-point iteration, with starting iterate $u_{n+1/2}^{(0)}=u_n$,
$$
\frac{u_{n+1/2}^{(k+1)} - u_n}{\tau/2} + A(u_{n+1/2}^{(k)}) u_{n+1/2}^{(k+1)} = f(u_{n+1/2}^{(k)}).
$$
If this iteration converges to a limit $u_{n+1/2}$, then $u_{n+1}=2u_{n+1/2}-u_n$ solves \eqref{mpr} with (FI). We write the above iteration briefly as
$$
u_{n+1/2}^{(k+1)} = \Phi(u_{n+1/2}^{(k)}) .
$$
Let $B_{2R}:=\{ y\in Y\,:\, \| y \|\le 2R \}$, which is a closed set in $X$, as is stated (without proof) in \cite[Lemma 7.3]{Kato75}. [This follows from a duality and density argument: for $y\in Y$, $\| y\| = \sup_{0\ne v\in X} (y,v)/\|v\|_*$, where $\|\cdot\|_*$ is the norm on the dual $Y'$ and we use the Gelfand triple $Y\subset X \subset Y'$ with dense and continuous embeddings. With this formula for $\|y\|$ it follows that for every sequence $(y_n)$ in $B_{2R}$ that converges to $x\in X$ in the $X$-norm, also $x\in B_{2R}$.] Therefore, $B_{2R}$ is a complete metric space with the metric $d(v,w)=|v-w|$.

By the argument of the proof of Lemma~\ref{lem:Y-bound-li} we find that for all $v\in B_{2R}$,
$$
\| \Phi(v) \| \le (1+C_R\tau) \| u_n \| + \tau \gamma_R \le 2R
$$
for sufficiently small stepsize $\tau\le \tau_R$. Hence, $\Phi$ maps $B_{2R}$ into itself.

We now show that $\Phi$ is a contraction on $B_{2R}$ for sufficiently small~$\tau$.
For $v,\widetilde v\in B_{2R}$, let $w=\Phi(v)$ and $\widetilde w=\Phi(\widetilde v)$. Then,
$$
\frac{\widetilde w-w}{\tau/2} + A(\widetilde v)(\widetilde w-w) = -\bigl( A(\widetilde v)-A(v)\bigr) w + f(\widetilde v)-f(v).
$$
Taking the inner product with $\widetilde w-w$ and using conditions (K1)--(K4), we obtain
$$
\frac2\tau |\widetilde w-w|^2 \le |\widetilde w-w| \, ( L_{2R}^A \cdot 2R + \ell_{2R}^X) |\widetilde v-v|
$$
and hence
$$
|\Phi(\widetilde v)-\Phi(v)|=|\widetilde w-w| \le c_R\tau |\widetilde v-v|.
$$
Therefore, if $c_R\tau<1$, then $\Phi$ is a contraction on $B_{2R}$, and the result follows with the
Banach fixed-point theorem.
\qed
\end{proof}

\begin{lemma}\label{lem:pert-fi} The statement of Lemma~\ref{lem:pert-li} is also valid for the fully implicit midpoint rule, \eqref{mpr} with (FI).
\end{lemma}

\begin{proof} With Lemma~\ref{lem:Y-bound-fi} at hand, the result follows with the proof of Lemma~\ref{lem:pert-li}.
\qed
\end{proof}

\subsection{Consistency error}
We now choose the exact solution values $u_n^\star=u(t_n)$ in \eqref{mpr-star}, with
$\widehat u_{n+1/2}^\star = u_n^\star + \tfrac12 (u_n^\star-u_{n-1}^\star)$ in the case of the linearly implicit midpoint rule (except for $n=0$, where $\widehat u_{1/2}^\star=u(0)$), and with $\widehat u_{n+1/2}^\star = u_{n+1/2}^\star = \tfrac12 (u_{n+1}^\star+u_n^\star)$ in the case of the fully implicit midpoint rule. The defects $d_{n+1/2}$ in \eqref{mpr-star} are then the consistency errors and are bounded as follows.

\begin{lemma}\label{lem:cons-error-mpr}
Suppose that the exact solution $u$ of \eqref{Q} has the regularity $u\in C^3([0,T],Y)$ with $Su\in C^2([0,T],Y)$. Then, the consistency errors \eqref{mpr-star} of the linearly and fully implicit midpoint rule are bounded by
$$
\| d_{n+1/2} \| \le C\tau^2,
$$
where $C$ is independent of $n$ and $\tau$ with $0\le n\tau\le T-\tau$ (except for $n=0$ for the linearly implicit method, where $\| d_{1/2} \| \le C\tau$).
\end{lemma}

\begin{proof}
First we note that Taylor expansion of $u$ at $t_{n+1/2}= (n+1/2)\tau$ yields
\begin{align*}
\| \widehat u_{n+1/2}^\star - u(t_{n+1/2}) \| &\le c \max_{t_{n}\le t \le t_{n+1}} \| \ddot u(t) \| \cdot \tau^2 ,
\\
\| S u_{n+1/2}^\star - Su(t_{n+1/2}) \| &\le c \max_{t_{n-1}\le t \le t_{n+1}} \| S\ddot u(t) \| \cdot \tau^2 ,
\\
\Bigl\| \frac{u(t_{n+1})- u(t_n)}\tau - \dot u(t_{n+1/2}) \Bigr\| &\le c  \max_{t_{n}\le t \le t_{n+1}} \| \dddot u(t) \| \cdot \tau^2.
\end{align*}
We denote
\begin{align*}
&r_{n+1/2} := A(\widehat u_{n+1/2}^\star) u_{n+1/2}^\star - A(u(t_{n+1/2}))  u(t_{n+1/2})
\\
&\ = A(\widehat u_{n+1/2}^\star) \bigl(u_{n+1/2}^\star - u(t_{n+1/2})\bigr) +
\bigl(A(\widehat u_{n+1/2}^\star) - A(u(t_{n+1/2})) \bigr) u(t_{n+1/2}) .
\end{align*}
Using \eqref{S}, this becomes
\begin{align*}
r_{n+1/2} &= S^{-1} A(\widehat u_{n+1/2}^\star) S\bigl(u_{n+1/2}^\star - u(t_{n+1/2})\bigr)
\\
&\quad +
S^{-1} B(\widehat u_{n+1/2}^\star) S\bigl(u_{n+1/2}^\star - u(t_{n+1/2})\bigr)
\\
&\quad + S^{-1}\bigl(A(\widehat u_{n+1/2}^\star) - A(u(t_{n+1/2})) \bigr) Su(t_{n+1/2})
\\
&\quad +
S^{-1}\bigl(B(\widehat u_{n+1/2}^\star) - B(u(t_{n+1/2})) \bigr) Su(t_{n+1/2}).
\end{align*}
By conditions (K1)--(K3) with $R=\max _{0\le t \le T} \| u(t) \|$, this is bounded by
\begin{align*}
\| r_{n+1/2} \| &\le M_R^A \| S u_{n+1/2}^\star - Su(t_{n+1/2}) \| + M_R^B | S u_{n+1/2}^\star - Su(t_{n+1/2}) |
\\
&\quad + L_R^A |\widehat u_{n+1/2}^\star - u(t_{n+1/2})|\, \| Su(t_{n+1/2}) \|
\\
&\quad + L_R^B \|\widehat u_{n+1/2}^\star - u(t_{n+1/2})\|\, | Su(t_{n+1/2}) |.
\end{align*}
Moreover, $
\| f(\widehat u_{n+1/2}^\star) - f(u(t_{n+1/2})) \| \le \ell_R^Y  \|\widehat u_{n+1/2}^\star - u(t_{n+1/2})\|.
$
Since
\begin{align*}
d_{n+1/2} &= \Bigl(\frac{u(t_{n+1})- u(t_n)}\tau - \dot u(t_{n+1/2}) \Bigr)- r_{n+1/2}
\\
&\quad - \bigl(f(\widehat u_{n+1/2}^\star) - f(u(t_{n+1/2}))\bigr),
\end{align*}
the result follows with the above estimates.
\qed
\end{proof}

\subsection{Error bounds}

Combining the lemmas of this section, we obtain the following error bound.

\begin{theorem}\label{thm:conv-mpr} Let the conditions (K1)--(K4) be satisfied, and
suppose that the solution $u$ of \eqref{Q} has the regularity $u\in C^3([0,T],Y)$ with $Su\in C^2([0,T],Y)$. Then, there exists $\bar\tau>0$ such that for stepsizes $0<\tau\le\bar\tau$, the errors of the fully and linearly implicit midpoint rules \eqref{mpr} with (FI) and (LI), respectively, are bounded by
$$
\| u_n -u(t_n) \| \le C \tau^2,
$$
where $C$ is independent of $n$ and $\tau$ with $0\le n\tau\le T$.
\end{theorem}

\section{Implicit Runge--Kutta methods}

\subsection{Method formulation and properties}
For a given stepsize $\tau>0$, an $m$-stage implicit Runge--Kutta method applied to the quasi-linear equation \eqref{Q} determines solution approximations $u_n\approx u(t_n)$ and internal stages $U_{ni}$ by the equations\footnote{Here the dot is just a suggestive notation, not a time derivative.}
\begin{align} \label{rk-int}
U_{ni} &= u_n + \tau \sum_{j=1}^m a_{ij} \dot U_{nj}, \qquad i=1,\ldots,m,
\\
\label{rk-sol}
u_{n+1} &= u_n + \tau \sum_{i=1}^m b_i \dot U_{ni},
\end{align}
where
\begin{equation} \label{rk-Q}
\dot U_{ni} + A(U_{ni}) U_{ni} = f(U_{ni}), \qquad i=1,\ldots,m.
\end{equation}
In the following we consider the equation without a semilinear term ($f=0$) for ease of presentation, since the semilinear term causes no substantial problems in the analysis but just leads to longer formulas. As in the previous section, all results are however readily generalized to a semilinear term satisfying (K4).

The method is determined by its coefficient matrix $\AA = (a_{ij})$ and its  vector of weights $b=(b_i)$. The method has {\it stage order} $q$ if, with the nodes $c_i=\sum_{j=1}^m a_{ij}$,
$$
\sum_{j=1}^m a_{ij} c_j^{k-1} = \frac{c_i^k}k \quad\ (i=1,\ldots,m) \quad\text{ for } k=1,\ldots, q.
$$
We always assume that the quadrature formula with weights $b_i$ and nodes $c_i$ has at least the {\it quadrature order} $q+1$:
$$
\sum_{j=1}^m b_{j} c_j^{k-1} = \frac{1}k  \qquad\text{ for } k=1,\ldots, q+1.
$$

In the following we consider Runge--Kutta methods that have the following important properties:

{\it Algebraic stability.} \cite{BurB,Cro} The weights $b_i$ are positive, and the matrix with entries $b_ia_{ij}+b_ja_{ji}-b_ib_j$ is positive semidefinite.

{\it Coercivity.} \cite{CroR} The Runge--Kutta coefficient matrix $\AA = (a_{ij})$ is invertible, and there exist a positive diagonal matrix $\calD=\text{diag}(d_i)$ and $\alpha>0$ such that
\begin{equation}\label{rk-coerc}
v^T \calD \AA^{-1} v \ge \alpha\, v^T \calD v \quad\ \text{for all }\ v\in \R^m.
\end{equation}
Important families of methods satisfying these properties are the Gauss and Radau IIA methods with an arbitrary number of stages $m\ge 1$; see, e.g., \cite{DekV} and \cite[Chapter IV]{HairerWannerII}.
The $m$-stage Gauss and Radau IIA methods have stage order $m$, and have quadrature order $2m$ and $2m-1$, respectively.

\subsection{Existence and uniqueness of the numerical solution}

\begin{lemma}\label{lem:rk-existence}
Let conditions (K1)--(K3) hold and let the Runge--Kutta method satisfy the coercivity condition \eqref{rk-coerc}. For every $R>0$, there exists $\tau_R>0$ such that the following holds: If $u_n\in Y$ with $\| u_n \|\le R$, then for stepsizes $\tau\le\tau_R$ the Runge--Kutta equations \eqref{rk-int} and \eqref{rk-Q} have a unique solution in $Y^m$ with
\begin{align*}
 |U_{ni}| &\le C |u_n|,
\\
\|U_{ni}\| &\le C \|u_n\|,
\end{align*}
where $C$ depends only on the Runge--Kutta coefficients.
\end{lemma}

\begin{proof} We may assume $n=0$ and write $U_i, \dot U_i$ instead of $U_{ni},\dot U_{ni}$ for brevity. Similar to the proof of Lemma~\ref{lem:Y-bound-fi} the proof is based on constructing a contractive fixed-point iteration. Here we consider the map $\Phi: (V_i)_{i=1}^m \mapsto (W_i)_{i=1}^m$ defined by the linear Runge--Kutta equations
\begin{align}
\label{rk-W}
&W_{i} = u_0 + \tau \sum_{j=1}^m a_{ij} \dot W_{j} \qquad (i=1,\ldots,m),
\\
\label{rk-Wdot}
&\dot W_{i} + A(V_{i}) W_{i} = 0.
\end{align}
We will show that, for sufficiently small stepsizes $\tau$, the map $\Phi$ is well-defined and a contraction in the $X^m$-norm on
$$
B_{cR} = \Bigl\{ (V_i)_{i=1}^m \in Y^m \,:\, \sum_{i=1}^m d_i \|V_i\|^2 \le (cR)^2 \Bigr\},
$$
where $c$ is a constant depending only on the Runge--Kutta coefficients, which will be specified below. Again by \cite[Lemma~7.3]{Kato75}, $B_{cR}$ is a closed set in $X^m$.

(i) We first prove that $\Phi$ is a well-defined map from $B_{cR}$ to itself for sufficiently small stepsizes~$\tau$. For $V=(V_i)_{i=1}^m$, we write $\calA(V)=\text{diag}(A(V_i))$. The equations
\eqref{rk-W}--\eqref{rk-Wdot} for $W=(W_i)_{i=1}^m$ are then written compactly as
$$
\bigl(I_m \otimes I+\tau(\AA\otimes I)\calA(V)\bigr)W = \ones\otimes u_0
$$
with $\ones=(1,\ldots,1)^T\in\R^m$, or equivalently,
\begin{equation} \label{W-eq}
\bigl( \AA^{-1}\otimes I +\tau\calA (V)\bigr)W = (\AA^{-1}\ones)\otimes u_0.
\end{equation}
By conditions (K1) and \eqref{rk-coerc}, the linear operator
$$
 \AA^{-1}\otimes I - \alpha I_m\otimes I +\tau\calA(V)
$$
is m-accretive with respect to the inner product on $X^m$ given by
$(W,\widetilde W)_\calD = \sum_{i=1}^m d_i (W_i,\widetilde W_i)$ with the corresponding norm
$|W|_\calD = (W, W)_\calD^{1/2}$.
Hence, equation \eqref{W-eq} has a unique solution $W\in D(\calA(V))$, and
$$
|W|_\calD \le \alpha^{-1} \,| (\AA^{-1}\ones) \otimes u_0 |_\calD \le c_0 |u_0|,
$$
where $c_0$ depends only on the Runge--Kutta coefficients.
We now recall condition (K2), which yields, with $\calB(V)=\text{diag}(B(V_i))$ and $\calS=I_m\otimes S$,
$$
\calA(V) = \calS^{-1} \calA(V) \calS + \calS^{-1} \calB(V) \calS.
$$
Therefore, $Z\in X^m$ is a solution of
\begin{equation}\label{Z-eq}
\bigl((\AA^{-1}\otimes I)+\tau\calA(V)+\tau\calB(V)\bigr)Z = (\AA^{-1}\ones) \otimes (Su_0),
\end{equation}
if and only if $W= \calS^{-1} Z\in Y^m$ solves \eqref{W-eq}. In view of \eqref{B-bound},
the $X^m$ operator norm of $\tau\calB(V)$ is bounded by
$$
|\tau \calB(V)| \le \tau  M_{CR}^B,
$$
where $C=c/\min_i \sqrt{d_i}$.
For sufficiently small $\tau$, the operator norm of $\tau\calB(V)$ is therefore bounded by $\alpha/2$, and then equation \eqref{Z-eq} has a unique solution $Z\in D(\calA(V))$ and
$$
\| W \|_\calD= |Z|_\calD \le \frac1{\alpha-\tau C_1M_{CR}^B} \, |(\AA^{-1}\ones) \otimes Su_0|_\calD \le  \frac1{\alpha-\tau M_{CR}^B} \,c_0\|u_0\| \le cR,
$$
where $\| W \|_\calD^2 = \sum_{i=1}^m d_i \|W_i\|^2$ and $c=2c_0/\alpha$.

%
%
%

(ii) Finally we show that $\Phi:B_{cR}\to B_{cR}$ is a contraction with respect to the $X^m$-norm $|\cdot|_\calD$ for sufficiently small stepsizes $\tau$. Let $W_i$ be defined by \eqref{rk-W}--\eqref{rk-Wdot} and similarly $\widetilde W_i$ by the same equations with $V_i$ replaced by $\widetilde V_i$. We denote
$E_i=W_i-\widetilde W_i$ and $\dot E_i=\dot W_i-\dot{\widetilde W_i}$ so that
\begin{align*}
& E_i = \tau \sum_{j=1}^m a_{ij} \dot E_j
\\
& \dot E_i + A(V_i) E_i = - \bigl( A(V_i)-A(\widetilde V_i)\bigr) \widetilde W_i.
\end{align*}
This is rewritten as
$$
((\AA^{-1}\otimes I)+\tau\calA(V))E = G:= -  \tau \Bigl(\bigl( A(V_i)-A(\widetilde V_i)\bigr) \widetilde W_i \Bigr)_{i=1}^m.
$$
We thus have, in view of the m-accretivity of $\calA(V)$, of the Lipschitz bound \eqref{Lip-A} and the bound $\| \widetilde W \|_\calD \le cR$,
$$
|E|_\calD \le \alpha^{-1} |G|_\calD \le \alpha^{-1} \tau L_{CR}^A \,|V-\widetilde V|_\calD \, cR.
$$
This shows that $\Phi$ is a contraction for sufficiently small~$\tau$.
The stated result then follows from the Banach fixed-point theorem.
\qed
\end{proof}

\subsection{Stability}
\begin{lemma}\label{lem:rk-bound}
In addition to the conditions of Lemma~\ref{lem:rk-existence}, let the Runge--Kutta method be algebraically stable. For every $R>0$, there exist $\tau_R>0$ and $C_R>0$ such that the following holds: If $u_n\in Y$ with $\| u_n \|\le R$, then for stepsizes $\tau\le\tau_R$ the Runge--Kutta equations \eqref{rk-int}--\eqref{rk-Q} have a unique numerical solution $u_{n+1}\in Y$ with
\begin{align*}
 |u_{n+1}| &\le  |u_n|,
\\
\|u_{n+1}\| & \le (1+C_R\tau)\|u_n\|.
\end{align*}
\end{lemma}

\begin{proof} The proof follows closely the standard use of algebraic stability for contractive differential equations; see \cite{BurB,Cro} and, e.g., \cite[Section IV.12]{HairerWannerII}. We take again $n=0$ and write $U_i$ for $U_{ni}$.
By \eqref{rk-sol} we have
$$
|u_{1}|^2 = |u_0|^2 + 2\tau \sum_{i=1}^m b_i (u_0,\dot U_i) +
\tau^2 \sum_{i,j=1}^m b_i b_j (\dot U_i, \dot U_j).
$$
Expressing $u_0$ in the second term on the right-hand side by \eqref{rk-int}, we obtain
$$
|u_{1}|^2 = |u_0|^2 + 2\tau \sum_{i=1}^m b_i (U_i,\dot U_i) -
\tau^2 \sum_{i,j=1}^m (b_ia_{ij}+b_ja_{ji}-b_i b_j) (\dot U_i, \dot U_j).
$$
By algebraic stability, we thus have
$$
|u_{1}|^2 \le |u_0|^2 + 2\tau \sum_{i=1}^m b_i (U_i,\dot U_i).
$$
Since $b_i>0$ and
$(U_i,\dot U_i)=-(U_i,A(U_i)U_i)\le 0$ by (K1), we obtain the bound $ |u_{1}|^2 \le  |u_0|^2$.

For the bound in the $Y$-norm we obtain in the same way
$$
|Su_{1}|^2 \le |Su_0|^2 + 2\tau \sum_{i=1}^m b_i (SU_i,S\dot U_i).
$$
Here we note, using subsequently \eqref{rk-Q}, (K2), (K1) and Lemma~\ref{lem:rk-existence},
\begin{align*}
(SU_i,S\dot U_i) &= - (SU_i, SA(U_i)S^{-1}SU_i) =
- (SU_i, A(U_i)SU_i) - (SU_i, B(U_i)SU_i)
\\
& \le M_{CR}^B |SU_i|^2 = M_{CR}^B \|U_i\|^2 \le M_{CR}^B \, (CR)^2 \| u_0\|^2,
\end{align*}
and the result follows.
\qed
\end{proof}
%
%

Suppose that $u_n^\star \in Y$ and $U_{ni}^\star \in Y$ (which will later be taken as the exact solution values $u(t_n)$ and $u(t_n+c_i \tau)$) solve \eqref{rk-int}--\eqref{rk-sol} up to the defects $d_{n+1}\in Y$ and $D_{ni} \in Y$, for $0 \leq n\tau \leq T$:
\begin{align}
    \label{rk-int-star}
    U_{ni}^\star &= u_n^\star + \tau \sum_{j=1}^m a_{ij} \dot U_{nj}^\star + D_{ni}, \qquad i=1,\ldots,m,
    \\
    \label{rk-sol-star}
    u_{n+1}^\star &= u_n^\star + \tau \sum_{i=1}^m b_i \dot U_{ni}^\star + d_{n+1},
\end{align}
where
\begin{equation*}
    \dot U_{ni}^\star + A(U_{ni}^\star) U_{ni}^\star = 0, \qquad i=1,\ldots,m.
\end{equation*}

\begin{lemma}
\label{lem:pert-rk}
    In the above situation, suppose that for $0\le n\tau \le T$ and for $i=1,\ldots,m$,
    \begin{align*}
      & \|u_{n}^\star\| \le R, \ \ \|U_{ni}^\star\| \le R \quad\text{ and }\quad SU_{ni}^\star\in Y \text{ with }\ \| SU_{ni}^\star \| \le K.
          \end{align*}
    Then, there exist  $\overline\tau>0$ and $\delta>0$, which depend only on $K$, $R$, $T$  and the coefficients of the Runge--Kutta method, such that for stepsizes $\tau\le\overline\tau$ and perturbations satisfying
    $$
    \| u_0- u_0^\star\|^2 + \tau\sum_{0\le n\tau\le T} \Big( \sum_{i=1}^m \|SD_{ni}\|^2 + \Big\|\frac{d_{n+1}}{\tau}\Big\|^2 + \|S d_{n+1}\|^2\Big) \le \delta^2,
    $$
    the error satisfies, for $0\le n\tau \le T$,
    \begin{align*}
    \| u_n- u_n^\star \|^2 &\le C_Y  \Bigl(  \| u_0- u_0^\star\|^2 +\tau\sum_{k=0}^{n-1} \Big(\sum_{i=1}^m \|SD_{ki}\|^2 + \Big\|\frac{d_{k+1}}{\tau}\Big\|^2 + \|S d_{k+1}\|^2\Big) \Bigr),
    \\
    | u_n - u_n^\star |^2 &\le C_X  \Bigl(  | u_0- u_0^\star |^2 +\tau\sum_{k=0}^{n-1} \Big(\sum_{i=1}^m \|D_{ki}\|^2 + \Big|\frac{d_{k+1}}{\tau}\Big|^2 + \|d_{k+1}\|^2\Big) \Bigr) ,
    \end{align*}
    where $C_Y$  depends only on $K$, $R$ and $T$, and $C_X$ depends only on $R$
    and~$T$.
\end{lemma}

\begin{proof}
The proof is similar to those of Lemmas~\ref{lem:pert-cont} and \ref{lem:pert-li}.
We denote the errors by
$$
e_n= u_n-u_n^\star,\qquad E_{ni} = U_{ni}-U_{ni}^\star, \qquad \dot E_{ni} = \dot U_{ni}- \dot U_{ni}^\star .
$$
We begin with $n=0$, and we write $E_i,U_i,U_i^\star$ instead of $E_{0i},U_{0i},U_{0i}^\star$ and analogously for the corresponding quantities carrying a dot. By subtracting the original and perturbed Runge--Kutta equations  we obtain
\begin{align}
    \label{rk-int-err}
    E_{i} &= e_0 + \tau \sum_{j=1}^m a_{ij} \dot E_{j} - D_{i}, \qquad i=1,\ldots,m,
    \\
    \label{rk-sol-err}
    e_1 &= e_0 + \tau \sum_{i=1}^m b_i \dot E_{i} - d_1,
\end{align}
where
\begin{equation}
\label{rk-Q-err}
    \dot E_i + A(U_i) E_i = -\big( A(U_i)-A(U_i^\star) \big) U_i^\star, \qquad i=1,\ldots,m.
\end{equation}
By condition (K2), the latter equation is equivalent to
\begin{equation}
\label{S-multiplied-rk-Q-err}
    \begin{aligned}
        S \dot E_i + A(U_i) SE_i + B(U_i) SE_i =&\ -\big( A(U_i)-A(U_i^\star) \big) SU_i^\star \\
        &\ -\big( B(U_i)-B(U_i^\star) \big) SU_i^\star .
    \end{aligned}
\end{equation}
As in the proof of Lemma~\ref{lem:rk-bound}, using \eqref{rk-int-err}--\eqref{rk-sol-err} and algebraic stability we obtain
\begin{equation*}
    \|e_1\|^2 - \|e_0\|^2 \leq 2 \tau \sum_{i=1}^m (SE_i+SD_i,S \dot E_i) - 2(Se_0 + \tau \sum_{i=1}^m b_i S \dot E_i, Sd_1) + \|d_1\|^2 .
\end{equation*}
We estimate the first two terms on the right-hand side separately.
We write
 \begin{align*}
    &(SE_i+SD_i,S \dot E_i)
    \\
    &= - (SE_i , A(U_i) SE_i) - (A(U_i)^*SD_i , SE_i)
     - (SE_i+SD_i , B(U_i) SE_i)
     \\
    &\ -(SE_i+SD_i , ( A(U_i)-A(U_i^\star) ) SU_i^\star)
     -(SE_i+SD_i , ( B(U_i)-B(U_i^\star) ) SU_i^\star) ,
 \end{align*}
 and note that the adjoint operator $A(y)^*$ is bounded like $A(y)$ for $\|y\|\le CR$: using that $Y$ is dense in $X$, and condition (K2) and recalling that $S$ is an isometry between $Y$ and $X$,
\begin{align*}
&| A(y)^* w | = \sup_{0\ne v \in Y} \frac{(A(y)^* w,v)}{|v|} =  \sup_{0\ne v \in Y} \frac{( w,A(y)v)}{|v|}
\\
&= \sup_{0\ne v \in Y} \frac{( Sw,S^{-1}A(y)SS^{-1}v)}{|v|} =
\sup_{0\ne v \in Y} \frac{( Sw,A(y)S^{-1}v + B(y)S^{-1}v)}{|v|}
\\
&\le \sup_{0\ne v \in Y} \frac{|S w| \cdot (M_{CR}^A  + M_{CR}^B) \|S^{-1}v\|}{|v|} =
\| w \| \cdot (M_{CR}^A  + M_{CR}^B).
\end{align*}
Using the relation \eqref{S-multiplied-rk-Q-err} and the accretivity \eqref{accretive}, the bounds \eqref{B-bound}--\eqref{Lip-f-Y}, we therefore obtain, as long as $\| U_i \| \leq CR$ ($i=1,\dotsc,m$),
\begin{align*}
    (SE_i+SD_i,S \dot E_i)
    \leq &\ (M_{CR}^A  + M_{CR}^B) \|SD_i\| \|E_i\| \\
    &\ + \big( M_{CR}^B + L_{CR}^A K + L_{CR}^B R \big)\|E_i+D_i\| \|E_i\| \\
    \leq &\ C_{K,R} \|E_i\|^2 + C_{K,R} \|SD_i\|^2 ,
\end{align*}
where we also used the norm relation $\|D_i\|=|S D_i| \leq \|S D_i\|$.

The other term is estimated similarly:
\begin{align*}
    (Se_0 + \tau \sum_{i=1}^m b_i S \dot E_i, Sd_1) \leq \|e_0\| \|d_1\| + \tau \sum_{i=1}^m b_i (S \dot E_i, Sd_1) ,
\end{align*}
where the terms in the sum are bounded by
\begin{align*}
    (S \dot E_i, Sd_1) \leq &\ (M_{CR}^A  + M_{CR}^B) \|E_i\| \|S d_1\|
    + \big( L_{CR}^A K + L_{CR}^B R \big) \|E_i\| \|d_1\| \\
    \leq &\  \|E_i\|^2 +  C_{K,R} \|S d_1\|^2. 
\end{align*}
By combining these estimates we obtain
\begin{align*}
    (Se_0 + \tau \sum_{i=1}^m b_i S \dot E_i, Sd_1)
    \leq &\ \tau \|e_0\|^2 + \tau c_0 \sum_{i=1}^m \|E_i\|^2 \\
    &\ + \tau C_{K,R} \|S d_1\|^2 + \tau C_{K,R} \Big\|\frac{d_1}{\tau}\Big\|^2 .
\end{align*}
Altogether, we have
\begin{align*}
    \|e_1\|^2 - \|e_0\|^2 \leq &\ \tau \|e_0\|^2 + \tau C_{K,R} \sum_{i=1}^m \|E_i\|^2  \\
    &\ + \tau C_{K,R} \sum_{i=1}^m \|SD_i\|^2 + \tau C_{K,R} \Big\|\frac{d_1}{\tau}\Big\|^2 + \tau C_{K,R} \|S d_1\|^2.
\end{align*}
To estimate the terms $\|E_i\|^2$, we use the coercivity property \eqref{rk-coerc} of the Runge--Kutta method.
We use the notations of the proof of Lemma~\ref{lem:rk-existence} and $E=(E_1,\dotsc,E_m)^T$, $\dot E=(\dot E_1,\dotsc,\dot E_m)^T$ and $D=(D_1,\dotsc,D_m)^T$. We thus rewrite \eqref{rk-int} as
\begin{equation*}
    E = \ones \otimes e_0 + \tau (\AA \otimes I) \dot E - D .
\end{equation*}
We multiply both sides by $\calD\AA^{-1} \otimes S$, use \eqref{S-multiplied-rk-Q-err} and (K2), and then take the inner product with $\calS E$, where again $\calS=I_m \otimes S$. Using similar estimates as above we obtain
\begin{align*}
    (\calS E, (\calD\AA^{-1} \otimes I) \calS E)
    = &\ \tau (\calS E, \calS \dot E)_{\calD} + (\calS E, (\calD\AA^{-1} \otimes I) (\ones \otimes Se_0 - \calS D)) \\
    \leq &\ \tau \big( M_{CR}^B + L_{CR}^A mK  + L_{CR}^B mR \big) \|E\|_{\calD}^2  \\
    &\ + c_0 \|E\|_{\calD}(\|e_0\| + \|D\|) \\
    \leq &\ \tau C_{K,R} \|E\|_{\calD}^2 + c_0 \|E\|_{\calD}(\|e_0\| + \|D\|) ,
\end{align*}
where the constant $c_0$ only depends on the method.
Using the coercivity of the Runge--Kutta method on the left-hand side, an absorption (by choosing the stepsize to satisfy $\tau C_{K,R} \le \alpha/2$) and Young's inequality for the right-hand side yields the bound
\begin{equation*}
    \sum_{i=1}^m\|E_i\|^2 \leq c \Big(\|e_0\|^2 + \sum_{i=1}^m \|D_i\|^2 \Big) .
\end{equation*}
Finally, combining all estimates we obtain
\begin{align*}
    \|e_1\|^2 - \|e_0\|^2 \leq &\ \tau (1+C_{K,R}c) \|e_0\|^2  \\
    &\ +\tau C_{K,R} \sum_{i=1}^m \|SD_i\|^2 +  \tau C_{K,R} \Big\|\frac{d_1}{\tau}\Big\|^2 + \tau C_R \|S d_1\|^2 .
\end{align*}
The analogous estimate for $ \|e_{n+1}\|^2 - \|e_n\|^2 $ holds for all $n$ as long as $\|U_{ni}\|\le CR$. Summing over $n$ and applying the discrete Gronwall inequality, we obtain the stated error bound in the $Y$-norm.
Choosing $\delta$ so small that $C_Y\delta\le R$, the condition $\| U_{ni} \|\le CR$ then remains satisfied for $0\leq n\tau \le T$.
The $X$-norm error bound is obtained analogously, using \eqref{accretive} and \eqref{Lip-A}.
\qed
\end{proof}

\subsection{Convergence with the stage order plus 1}

Using $u_n^\star=u(t_n)$ and $U_{ni}^\star=u(t_n+c_i\tau)$ in Lemma~\ref{lem:pert-rk}, we obtain the following error bound.

\begin{theorem}\label{thm:conv-rk-stage-order} Let the conditions (K1)--(K4) be satisfied, and
suppose that the solution $u$ of \eqref{Q} has the regularity $u\in C^{q+2}([0,T],Y)$ with $Su\in C^{q+1}([0,T],Y)$. Then, there exists $\bar\tau>0$ such that for stepsizes $0<\tau\le\bar\tau$, the errors of an algebraically stable and coercive Runge--Kutta method with stage order $q$  and quadrature order at least $q+1$ are bounded by
$$
\| u_n -u(t_n) \| \le C \tau^{q+1},
$$
where $C$ is independent of $n$ and $\tau$ with $0\le n\tau\le T$.
\end{theorem}

\begin{proof} With the choice $u_n^\star=u(t_n)$ and $U_{ni}^\star=u(t_n+c_i\tau)$ in Lemma~\ref{lem:pert-rk}, the defects $D_{ni}$ and $d_{n+1}$ in \eqref{rk-int-star} and \eqref{rk-sol-star}
are just quadrature errors:
\begin{align*}
D_{ni} &= \tau^{q} \int_{t_n}^{t_{n+1}} \kappa_i\Bigl(\frac{t-t_n}\tau\Bigr) u^{(q+1)}(t)\, \d t ,
\\
d_{n+1} &= \tau^{q+1} \int_{t_n}^{t_{n+1}} \kappa\Bigl(\frac{t-t_n}\tau\Bigr) u^{(q+2)}(t)\, \d t =
- \tau^{q} \int_{t_n}^{t_{n+1}} \kappa'\Bigl(\frac{t-t_n}\tau\Bigr) u^{(q+1)}(t)\, \d t
\end{align*}
with real-valued, bounded Peano kernels $\kappa_i$ and $\kappa$. The result then follows from Lemma~\ref{lem:pert-rk}.
\qed
\end{proof}

\subsection{Convergence with the classical order}
A Runge--Kutta method has {\it classical order} $p$ if the local error (i.e., the error after one step starting from the exact solution) is of size $\bigo(\tau^{p+1})$ whenever the method is applied to an ordinary differential equation $\dot y=f(y)$ in $\R^n$ with an arbitrarily differentiable function $f$. We recall that the classical order of the $m$-stage Gauss and Radau IIA methods is $2m$ and $2m-1$, respectively, whereas the stage order of these methods is $m$; see \cite[Chapter IV]{HairerWannerII}.

We now show that for the quasi-linear problem \eqref{Q} we can retain the classical order under additional regularity conditions.
The first such condition is a generalization of condition (K2):

\medskip\noindent
For $k=1,\dots,p-q$ and for every $y\in Y$ with $\|y\|\le R$,
\begin{equation}
\label{S-k}
S^kA(y)S^{-k} = A(y) + B_k(y),
\end{equation}
where $B_k(y)$ is $Y$-locally uniformly bounded on $X$:
\begin{equation}
\label{B-bound-k}
|B_k(y)v| \le M_{k,R} \, |v| \qquad\text{for all }\ v\in X.
\end{equation}

\medskip
With $L(Y,X)$ denoting the Banach space of bounded linear operators from $Y$ to $X$ (and analogously $L(X,X)$), we further suppose that the operators $A$ and $B$ of \eqref{S} satisfy the following:

\medskip
 $A(\cdot):\, y\in Y \mapsto A(y) \in L(Y,X)$ and $B_k(\cdot): y\in Y \mapsto B(y) \in L(X,X)$ are arbitrarily differentiable, and for any $R>0$, their derivatives up to any fixed order are uniformly bounded for $\| y \| \le R$.
\medskip

In the presence of a semilinear term $f(y)$ in \eqref{Q} (which we have discarded in this section), a similar $Y$-locally uniform differentiability condition is required for $f$.
We note that the above conditions are satisfied in all the examples of \cite{Kato75}.

The following theorem can be viewed as an extension of the full-order error bounds for {\it linear} evolution equations in \cite{Cro-these,LubO-int,Mansour} to the quasi-linear case studied here.

\pagebreak[2]

\begin{theorem}\label{thm:conv-rk-classical-order} Let  (K1)--(K4) and the above conditions be satisfied, and
suppose that the solution $u$ of \eqref{Q} has the regularity $u\in C^{p+1}([0,T],Y)$ with $S^ku\in C^{p+1-k}([0,T],Y)$ for $k=1,\ldots,p-q$. Then, there exists $\bar\tau>0$ such that for stepsizes $0<\tau\le\bar\tau$, the errors of an algebraically stable and coercive Runge--Kutta method with stage order~$q$  and classical order~$p$ (with $2q\ge p$) are bounded by
$$
\| u_n -u(t_n) \| \le C \tau^{p},
$$
where $C$ is independent of $n$ and $\tau$ with $0\le n\tau\le T$.
\end{theorem}

\noindent
{\it Remark\/} The condition $2q\ge p$ simplifies the proof and is satisfied for the Gauss and Radau IIA methods, which are arguably the most interesting classes of implicit Runge--Kutta methods. We expect, however, that this condition can be dropped.

\begin{proof} (a)
Let us first show how we get from order of convergence $q+1$ to order $q+2$. We start by taking as  $U_{ni}^\star$ in \eqref{rk-int-star} the exact solution value at $t_n+c_i\tau$. In the following we can again take $n=0$ and drop the dependence on $n$ in the notation.
We then modify the reference internal stages by setting
$$
U_{i}^{[q+2]} = U_{i}^\star - D_{i}
$$
and $\dot U_{i}^{[q+2]} = - A(U_{i}^{[q+2]})U_{i}^{[q+2]}$. The modified defect $D_{i}^{[q+2]}$ in
the Runge--Kutta equations,
$$
U_{i}^{[q+2]}  = u(t_0) + \tau \sum_{j=1}^m a_{ij} \dot U_{j}^{[q+2]} + D_{i}^{[q+2]},
$$
is then
$$
D_{i}^{[q+2]} = \tau \sum_{j=1}^m a_{ij} \bigl( A(U_{j}^\star)U_{j}^\star - A(U_{j}^\star-D_j)(U_{j}^\star-D_j)\bigr),
$$
which by (K1)--(K3), by the Peano kernel formula for $D_i$ in the proof of Theorem~\ref{thm:conv-rk-stage-order} and the regularity assumption for the exact solution is bounded by
\begin{align*}
&\| D_{i}^{[q+2]} \| = | SD_{i}^{[q+2]} |
\\
&\le \tau \sum_{j=1}^m |a_{ij}| \bigl|S\bigl( A(U_{j}^\star)D_{j} + (A(U_{j}^\star)- A(U_{j}^\star-D_j))(U_{j}^\star-D_j)\bigr)\bigr|
\\
&\le  \tau \sum_{j=1}^m |a_{ij}|  \bigl| A(U_{j}^\star)SD_{j} + B(U_{j}^\star)SD_{j} +
(A(U_{j}^\star)- A(U_{j}^\star-D_j))(SU_{j}^\star-SD_j) \bigr.
\\
&\qquad\qquad\qquad+ \bigl. (B(U_{j}^\star)- B(U_{j}^\star-D_j))(SU_{j}^\star-SD_j) \bigr|
\\
&\le
c \,\tau\,\max_{1\le j \le m} \| SD_j \| \le C\tau^{q+2}.
\end{align*}
The defect in
$$
u(t_1)  = u(t_0) + \tau \sum_{i=1}^m b_i \dot U_{i}^{[q+2]} + d_1^{[q+2]}
$$
is then
$$
d_1^{[q+2]} = d_1+\tau \sum_{i=1}^m b_i \bigl( A(U_{i}^\star)U_{i}^\star - A(U_{i}^\star-D_i)(U_{i}^\star-D_i)\bigr),
$$
where we know already that
$$
\| d_1 \| \le C \tau^{p+1} \le C \tau^{q+3}
$$
in the case of interest where $p\ge q+2$. The more challenging term is
\begin{align*}
&\tau \sum_{i=1}^m b_i \Bigl( A(U_{i}^\star)U_{i}^\star - A(U_{i}^\star-D_i)(U_{i}^\star-D_i)\Bigr)
\\
&=
\tau \sum_{i=1}^m b_i \Bigl( A(U_{i}^\star)D_i - \int_0^1 A'(U_{i}^\star-\theta D_i)[D_i](U_{i}^\star-D_i)\, \d\theta \Bigr).
\end{align*}
This differs by $\bigo(\tau^{q+3})$ in the $Y$-norm from
$$
\tau \sum_{i=1}^m b_i \bigl( A(u(t_0))D_i -  A'(u(t_0))[D_i]u(t_0) \bigl) =0,
$$
because we have the quadrature error
$$
D_i = \tau^{q+1} \Bigl( \sum_{i=1} a_{ij} c_j^{q} -\frac {c_i^{q+1}}{q+1} \Bigr) u^{(q+1)}(t_0) + \bigo(\tau^{q+2}) ,
$$
and the order conditions for the $p$th-order Runge--Kutta method (see \cite{Butcher-book,HNW}) yield
$$
\sum_{i=1}^m b_i \Bigl( \sum_{i=1}^m a_{ij} c_j^{q} -\frac {c_i^{q+1}}{q+1} \Bigr) =0.
$$
Hence,
$$
\| d_1^{[q+2]}  \| \le C\tau^{q+3}.
$$
Lemma~\ref{lem:pert-rk} used with $U_{i}^{[q+2]}$ in the role of $U_i^\star$ then yields the result for $p=q+2$.

(b) The above procedure for modifying the reference internal stages can be repeated. In the next step we set
$$
U_{i}^{[q+3]} = U_{i}^{[q+2]} - D_{i}^{[q+2]}
$$
and so on we iterate up to $U_i^{[p]}$. Under the given regularity conditions we then obtain defects
$$
U_{i}^{[p]}  = u(t_0) + \tau \sum_{j=1}^m a_{ij} \dot U_{j}^{[p]} + D_{i}^{[p]}
$$
with
$$
\| D_{i}^{[p]} \| \le C\tau^{p},
$$
gaining a factor $\tau$ at the expense of an application of $S$ to the previous defect in every iteration. The defect $d_1^{[p]}$ in
$$
u(t_1)  = u(t_0) + \tau \sum_{i=1}^m b_i \dot U_{i}^{[p]} + d_1^{[p]}
$$
then becomes a more complicated expression than before, but the key observation is that it can be Taylor-expanded into terms of the form
$$
\tau^k \sum_{i,j_1,\ldots,j_r=1}^m b_i c_i^{\ell_0} a_{ij_1} c_{j_1}^{\ell_1} \ldots a_{j_{r-1}j_{r}} c_{j_{r}}^{\ell_r}
\Bigl( \sum_{j=1}^m a_{j_{r}j} c_j^{s-1} -\frac {c_{j_r}^{s}}{s} \Bigr)
$$
multiplied with an expression depending on the solution $u$ and its derivatives evaluated at~$t_0$.
We omit the details. For $k\le p$ these terms all vanish by the order conditions of the Runge--Kutta method \cite{Butcher-book,HNW}. In this way we obtain
$$
\| d_1^{[p]} \| \le C \tau^{p+1} \quad \text{ and } \quad \| S d_1^{[p]} \| \le C \tau^{p},
$$
and the result then follows again by Lemma~\ref{lem:pert-rk} with $U_i^{[p]}$ in the role of $U_i^\star$.
\qed
\end{proof}

\section*{Acknowledgement}
This work was supported by Deutsche Forschungsgemeinschaft, SFB 1173.


\providecommand{\bysame}{\leavevmode\hbox to3em{\hrulefill}\thinspace}
\providecommand{\MR}{\relax\ifhmode\unskip\space\fi MR }
\providecommand{\MRhref}[2]{%
  \href{http://www.ams.org/mathscinet-getitem?mr=#1}{#2}
}
\providecommand{\href}[2]{#2}

\end{document}